\documentclass[a4paper,11pt]{article}
\usepackage{enumerate}
\usepackage{amsmath,amsthm,amssymb}

\title{Directed Simplices In Higher Order Tournaments}
\author{Imre Leader \thanks{Department of Pure Mathematics and Mathematical Statistics, Centre for Mathematical Sciences, University of Cambridge, Wilberforce Road, Cambridge CB3 0WB, United Kingdom. Email: I.Leader@dpmms.cam.ac.uk.} \and Ta Sheng Tan \thanks{Department of Pure Mathematics and Mathematical Statistics, Centre for Mathematical Sciences, University of Cambridge, Wilberforce Road, Cambridge CB3 0WB, United Kingdom. Email: T.S.Tan@dpmms.cam.ac.uk.}}
\newtheorem{thm}{Theorem}[section]
\newtheorem{lemma}[thm]{Lemma}
\newtheorem{cor}[thm]{Corollary}
\newtheorem{qs}{Question}
\theoremstyle{remark} \newtheorem*{remark}{Remark}
\begin{document}

\maketitle

\begin{abstract}
It is well known that a tournament (complete oriented graph) on $n$ vertices has at most $\frac{1}{4}\binom{n}{3}$ directed triangles, and that the constant $\frac{1}{4}$ is best possible. Motivated by some geometric considerations, our aim in this paper is to consider some `higher order' versions of this statement. For example, if we give each $3$-set from an $n$-set a cyclic ordering, then what is the greatest number of `directed $4$-sets' we can have? We give an asymptotically best possible answer to this question, and give bounds in the general case when we orient each $d$-set from an $n$-set.
\end{abstract}

\section{Introduction}

A \textit{tournament} is a complete graph in which each edge is assigned a direction. It is well known (see e.g. \cite{moon}) that there are at most $\frac{1}{4}\binom{n}{3}+O(n^2)$ directed triangles in a tournament on $n$ vertices. The constant $\frac{1}{4}$ is easily seen to be best possible, since for example the random tournament (where the direction of each edge is assigned randomly and independently with probability $\frac{1}{2}$) has expected number of directed triangles precisely $\frac{1}{4}\binom{n}{3}$. Actually, any tournament in which all degrees are close to $\frac{n}{2}$ will have about this number of directed triangles (see e.g. \cite{moon}).
\\\\
Our aim in this paper is to investigate some `higher order' analogues of this result. Before we make our definitions, we give some geometric background, to explain how the question arose. However, our question is natural even without any motivation, so the reader who is not interested in geometric considerations should feel free to skip the next few paragraphs.
\\\\
Let $T \subset \mathbb{R}^d$ be a set of $n$ points in general position. What is the greatest possible number of $d$-simplices of $T$ that contain (say in their interior) a given point of $\mathbb{R}^d$? In two dimensions, this question was asked by K\'arteszi \cite{karteszi} and answered by Boros and F\"uredi \cite{boros1,boros2}, who showed that for any set $T$ of $n$ points in the plane in general position and any point $x$ the number of triangles of $T$ containing $x$ is at most $\frac{1}{4}\binom{n}{3}+O(n^2)$. (Note that this can be attained, for example by taking $T$ to be a regular $n$-gon and $x$ its centre). Their elegant proof was to note that there is a natural way to make $T$ into a tournament: given $a$ and $b$ in $T$, direct the edge $ab$ from $a$ to $b$ (respectively from $b$ to $a$) in such a way that the triple $abx$ (respectively $bax$) is clockwise. Then the triangles of $T$ containing $x$ correspond precisely to the directed triangles of this tournament. 
\\\\
In this paper we are usually interested in asymptotic bounds, but we remark in passing that Boros and F\"uredi actually proved the exact best possible bound on the number of triangles, because the exact tournament bound (namely $\frac{1}{24}(n^3-n)$ if $n$ is odd and $\frac{1}{24}(n^3-4n)$ if $n$ is even) can in fact be realised geometrically. Indeed, the above construction, with $x$ moved slightly so as not to be collinear with any pair from $T$, achieves this value.
\\\\
The general question (in $d$ dimensions) was asked by Boros and F\"uredi, and answered by B\'ar\'any \cite{barany}. He showed that if $T \subset \mathbb{R}^d$ is a set of $n$ points in general position and $x$ is any point then the number of $d$-simplices of $T$ containing $x$ is at most $\frac{1}{2^d} \binom{n}{d+1} + O(n^d)$. The constant $\frac{1}{2^d}$ is best possible, as may be seen in \cite{barany}. 
\\\\
(The reader may like to note that, while the above is in some sense about the `best' sets $T$, it is also very natural to ask about the `worst' sets. Thus for example Boros and F\"uredi \cite{boros2} (see also Bukh \cite{bukh1}) showed that for any set $T$ of $n$ points in general position there is a point in at least $\frac{2}{9}$ of its triangles, and the constant $\frac{2}{9}$ cannot be improved. In $d$ dimensions, the right constant is not known: the best current bounds are a lower bound of $\frac{1}{(d+1)^d}$ by B\'ar\'any \cite{barany}, improved slightly to $\frac{d^2+1}{(d+1)^{d+1}}$ by Wagner \cite{wagner}, and an upper bound of $\frac{d!}{(d+1)^d}$ by Bukh, Matou\v sek and Nivasch \cite{bukh2}.)
\\\\
Now, B\'ar\'any's result uses the Upper Bound Theorem \cite{mullen} (about facet counts in polytopes). In other words, it uses a geometric theorem, as opposed to the abstract tournament theorem used by Boros and F\"uredi. But what would the corresponding abstract result be? Just as in the case $d=2$, for a general $d$ we would give an orientation to each $d$-set ($(d-1)$-simplex) in $T$, according to `on which side of it' the point $x$ lies. And then the $d$-simplices containing $x$ would correspond exactly to the $(d+1)$-sets in $T$ whose $d$-sets were `oriented compatibly' (in other words, whose $d$-sets had orientations that could be induced from a fixed orientation of the $(d+1)$-set -- this will be made more precise in a moment). Hence our abstract question is as follows: suppose that we orient (in some sense) every $d$-set of an $n$-set; what is the greatest number of directed $(d+1)$-sets that arise? In particular, do we get as small a bound as $\frac{1}{2^d} \binom{n}{d+1} + O(n^d)$? 
\\\\
We now give the precise (and non-geometric) definitions. We define an \textit{orientation} of a $d$-set inductively. An orientation of a 1-set $\{x\}$ is just an assignment of $\pm1$ to $x$, and an orientation of a 2-set $\{a,b\}$ is a directed edge from $a$ to $b$ or vice versa. (We may, if we wish, think of a directed edge from $a$ to $b$ as assigning $+1$ to $b$ and $-1$ to $a$). And for $d\geq3$, an orientation of a $d$-set consists of an orientation for each of its $(d-1)$-subsets in such a way that these orientations are \textit{compatible}, meaning that any two give different orientations to their common $(d-2)$-subset. Then, for $d \geq 2$, a \textit{tournament of order $d$}, or \textit{$d$-tournament}, consists of a set together with an orientation of each of its $d$-sets. Finally, in a $d$-tournament a \textit{$d$-simplex} is a $(d+1)$-set, and we say that it is \textit{directed} if its $d$-subsets are pairwise compatible.
\\\\
For example, a 2-tournament is just a tournament, and its directed $3$-sets are precisely its directed triples in the usual sense. And a 3-tournament is specified by giving each 3-set (from a given set) a cyclic ordering: then a 4-set is directed if any two of its 3-sets have cyclic orderings that go in opposite directions on their common 2-set.
\\\\
Our question is then: what is the greatest number of directed $(d+1)$-sets for a $d$-tournament on $n$ vertices? For $d=2$ this is $\frac{1}{4}\binom{n}{3}+O(n^2)$; what can we say in general? And how does this bound compare with the `geometric' version (when the $d$-tournament is induced from a set $T$ in $\mathbb{R}^d$), where the bound is $\frac{1}{2^d} \binom{n}{d+1} + O(n^d)$? 
\\\\
To put it another way, define the constant $c_d$ to be the limit, as $n \rightarrow \infty$, of this greatest number as a fraction of $\binom{n}{d+1}$ -- an easy averaging argument shows that the limit does exist. In this language, the $d=2$ result is that $c_2=\frac{1}{4}$, and the geometric construction shows that $c_d \geq \frac{1}{2^d}$. In fact, another reason why it is obvious that $c_d \geq \frac{1}{2^d}$ is that a random $d$-tournament has expected number of directed $(d+1)$-sets exactly $\frac{1}{2^d}\binom{n}{d+1}$. How does $c_d$ behave, for fixed small $d$ and also as $d$ gets large?
\\\\
The plan of the paper is as follows. We start by considering the case $d=3$. Here it turns out that $\frac{1}{8}$ is not the right answer. We give an upper bound of $\frac{1}{4}$, by a simple counting argument. And then we show that that in fact $c_3=\frac{1}{4}$, by a slightly unexpected random argument. This is the content of Section 2.
\\\\
Then in Section 3 we turn our attention to general $d$. Here we do not know what the exact value of $c_d$ is. We give an upper bound of $\frac{1}{d+1}$, again by a simple counting argument. For the lower bound, the method for $d=3$ seems unfortunately not to generalise, and indeed we do not know how to use any random methods to improve significantly on $\frac{1}{2^d}$. However, we give an explicit construction to show that $c_d \geq \frac{1}{d^2}$. Thus the abstract version of the problem exhibits genuinely different behaviour to the geometric version.
\\\\
In Section 4, we give some remarks and open questions.

\section{Directed Tetrahedra in Tournaments of Order 3}

In this section, we determine $c_3$. In a 3-tournament, we call a 3-set a \textit{triangle} and a 3-simplex (or a 4-set) a \textit{tetrahedron}. So a directed tetrahedron is simply a directed 4-set. 
\\\\
Given a triangle $\triangle =\{a,b,c\}$, it can be oriented (in a 3-tournament) either as

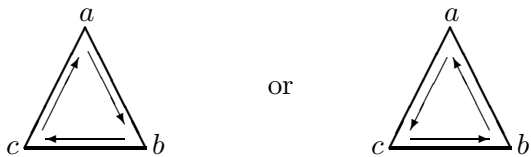
\setlength{\unitlength}{0.8cm}
\begin{picture}(10,3)
  \thicklines
  \put(1,0.5){\line(1,0){2}}
  \put(2,2.5){\line(-1,-2){1}}
  \put(2,2.5){\line(1,-2){1}}
  \put(1.9,2.6){$a$}
  \put(0.7,0.4){$c$}
  \put(3.1,0.4){$b$}
  \put(5,1.4){\text{or}}
  \put(7,0.5){\line(1,0){2}}
  \put(8,2.5){\line(-1,-2){1}}
  \put(8,2.5){\line(1,-2){1}}
  \put(7.9,2.6){$a$}
  \put(6.7,0.4){$c$}
  \put(9.1,0.4){$b$}
  \thinlines
  \put(2.65,0.65){\vector(-1,0){1.3}}
  \put(2.05,2.1){\vector(1,-2){0.6}}
  \put(1.3,0.8){\vector(1,2){0.6}}

  \put(7.35,0.65){\vector(1,0){1.3}}
  \put(8.65,0.8){\vector(-1,2){0.6}}
  \put(7.95,2){\vector(-1,-2){0.6}}
\end{picture}
\\
We write $\overrightarrow{abc}$ $\left(= \overrightarrow{bca} \text{ or } \overrightarrow{cab}\right )$ for the former and $\overrightarrow{acb}$ $\left(= \overrightarrow{bac} \text{ or } \overrightarrow{cba}\right )$ for the latter. Thus, in this language, a directed tetrahedron $\{a,b,c,d\}$ in a 3-tournament has the orientations of its triangles as $\left\{\overrightarrow{abc},\overrightarrow{adb},\overrightarrow{acd},\overrightarrow{bdc}\right\}$ or $\left\{\overrightarrow{acb},\overrightarrow{abd},\overrightarrow{adc},\overrightarrow{bcd}\right\}$.
\\\\
It is easy to check (by hand) that in a tetrahedron there are at least 2 compatible pairs of triangles.

\begin{thm}\label{uppertetrahedron}
Let $T_3$ be a 3-tournament on $n$ vertices. Then the number of directed tetrahedra in $T_3$ is at most $\frac{1}{4}\binom{n}{4}+O(n^3)$.
\end{thm}

\begin{proof}
Let $X$ be the number of directed tetrahedra in $T_3$. For each directed tetrahedron, there are 6 compatible pairs of triangles. We also know that there are at least 2 compatible pairs of triangles in each of the non-directed tetrahedron. Therefore, there are at least $6X+2\left(\binom{n}{4}-X\right)$ compatible pairs of triangles in $T_3$.
\\\\
On the other hand, consider any 2-set $\{a,b\}$ and count the number of compatible pairs of triangles having $\{a,b\}$ as their common 2-subset. By considering the orientation of $\{a,b,x\}$ for each $x\notin \{a.b\}$, it is easy to see that this number is at most $\lfloor\frac{n-2}{2}\rfloor\lceil\frac{n-2}{2}\rceil$. Therefore, there are at most $\lfloor\frac{n-2}{2}\rfloor\lceil\frac{n-2}{2}\rceil \binom{n}{2}$ compatible pairs of triangles in $T_3$.
\\\\
Putting the two bounds together, we have 
\begin{align*}
6X+2\left(\binom{n}{4}-X\right) &\leq \Big\lfloor\frac{n-2}{2}\Big\rfloor\Big\lceil\frac{n-2}{2}\Big\rceil \binom{n}{2}\\
\text{so } \qquad \qquad                         X &\leq \frac{1}{4}\binom{n}{4}+O(n^3).
\end{align*}
\end{proof}

\noindent A natural guess for $c_3$ would be $\frac{1}{8}$. This is from looking at the geometric version of the problem or by assigning the orientation of each triangle randomly and independently with probability $\frac{1}{2}$. But surprisingly this is not the case, as shown in the following construction.
\\\\
The idea is to construct a 3-tournament $T_3$ by inducing it in a certain way from a random 2-tournament. We will show that there is a way of inducing which gives $T_3$ many directed tetrahedra.
\\\\
Suppose $T_2$ is a 2-tournament on $n$ vertices. Let $T_3$ be a 3-tournament with vertex set $V(T_2)$. Note that there are only two types of 2-tournament on three vertices; either a directed triangle or a transitive 3-set. Let $\triangle=\{a,b,c\}$ be a triangle. If $\triangle$ is a directed triangle (in $T_2$) with directed edges $\{a\rightarrow b, b\rightarrow c, c\rightarrow a\}$, we will, in $T_3$, orient $\triangle$ the ``same'' way, $\overrightarrow{abc}$. If $\triangle$ is a transitive 3-set with directed edges $\{a\rightarrow b, b\rightarrow c, a\rightarrow c\}$, we could choose to orient $\triangle$ following the ``minority'', $\overrightarrow{acb}$, or the ``majority'', $\overrightarrow{abc}$. It turns out that it is better to orient the transitive 3-sets following the ``minority''.

\begin{thm} \label{lowertetrahedron}
For $n\geq 4$, there is a 3-tournament on $n$ vertices that has at least $\frac{1}{4}\binom{n}{4}$ directed tetrahedra.
\end{thm}

\begin{proof}
Let $T_2$ be a random tournament on $n$ vertices where the direction of each edge is assigned independently with probability $\frac{1}{2}$. Let $T_3$ be the 3-tournament induced from $T_2$ where the orientation of each triangle is assigned according to the preceding discussion (that is, orient a directed triangle the ``same'' way and a transitive 3-set following the ``minority''). 
\\\\
Given a set of four vertices $D=\{a,b,c,d\}$, it is easy to check that $D$ is a directed tetrahedron in $T_3$ if and only if $D$ is one of the following two types of tournaments in $T_2$:
\begin{enumerate}[(i)]
\item a vertex dominating a directed triangle, $\{t \rightarrow x, t \rightarrow y, t \rightarrow z, x \rightarrow y, y \rightarrow z, z\rightarrow x \}$, or
\item a vertex dominated by a directed triangle, $\{x \rightarrow t, y \rightarrow t, z \rightarrow t, x \rightarrow y, y \rightarrow z, z\rightarrow x \}$,
\end{enumerate}
for $\{t,x,y,z\}=\{a,b,c,d\}$.
\\
Letting $X$ be the total number of type (i) and type (ii) tournaments in $T_2$, it follows that 
\[
\mathbb{E}(\text{number of directed tetrahedra in $T_3$})=\mathbb{E}(X).
\]
Now, there are $2^6=64$ different 2-tournaments on $\{1,2,3,4\}$, of which there are $8$ type (i) tournaments and 8 type (ii) tournaments. So,
\[ \mathbb{E}(X)=\frac{8+8}{64} \binom{n}{4} = \frac{1}{4} \binom{n}{4}.\]
Therefore, there exists a 2-tournament $T_2$ with at least $\frac{1}{4} \binom{n}{4}$ type (i) and type (ii) tournaments. Hence, there is a 3-tournament that has at least $\frac{1}{4}\binom{n}{4}$ directed tetrahedra.
\end{proof}

\noindent Putting Theorem~\ref{uppertetrahedron} and Theorem~\ref{lowertetrahedron} together, we have the following corollary.
\begin{cor}
$c_3=\frac{1}{4}.$ \qed
\end{cor}

\section{Tournaments of Higher Order}

In this section, we consider tournaments of order $d$ for a general $d\geq 2$.
\\\\
We first show that in a $d$-tournament, at most a fraction of $\frac{1}{d+1}$ of the $d$-simplices are directed. This uses a very similar counting argument as the case $d=3$.
\\\\
The following lemma gives a lower bound on the number of compatible pairs of $d$-sets in a $d$-simplex.

\begin{lemma} \label{thelemma}
Fix $d\geq 2$. Let $S$ be a $d$-simplex. Then $S$ has at least $s(d)$ compatible pairs of $d$-sets, where
\[
s(d)=\binom{\lfloor\frac{d+1}{2}\rfloor}{2}+\binom{\lceil\frac{d+1}{2}\rceil}{2}.
\]
\end{lemma}
\begin{proof}
The $d$-simplex $S$ has $d+1$ oriented $d$-sets. Suppose that switching the orientations of $x$ $d$-sets turns $S$ into a directed $d$-simplex. Then we can partition the $d$-sets of $S$ as $S'\cup S''$, where $S'$ is the set of $d$-sets for which we need to switch orientations. Consider two $d$-sets in $S'$; switching the orientation of both the $d$-sets results in them being compatible by the definition of a directed simplex. Hence they must be compatible before switching. For a similar reason, a $d$-set in $S'$ and a $d$-set in $S''$ cannot be a compatible pair. Also, any pair of $d$-sets in $S''$ are compatible. So the number of compatible pairs of $d$-sets in $S$ is
\[ \binom{x}{2}+\binom{d+1-x}{2}\]
which is minimised at $x=\lfloor\frac{d+1}{2}\rfloor$.
\end{proof}
\noindent Based on this, we present an upper bound for $c_d$.

\begin{thm} \label{upperbound}
Fix $d\geq 2$. Let $T_d$ be a $d$-tournament on $n$ vertices. Then the number of directed $d$-simplices in $T_d$ is at most
\[
\left\{
\begin{array}{rl}
\frac{1}{d+1}\binom{n}{d+1}+O(n^d) & \text{if $d$ is odd}\\ \\
\frac{1}{d+2}\binom{n}{d+1}+O(n^d) & \text{if $d$ is even.}
\end{array} \right.
\]
In particular, 
\[
c_d\leq \left\{
\begin{array}{rl}
\frac{1}{d+1} & \text{if $d$ is odd}\\ \\
\frac{1}{d+2} & \text{if $d$ is even.}
\end{array} \right.
\]
\end{thm}

\begin{proof}
Let $X$ be the number of directed $d$-simplices in $T_d$. For each directed $d$-simplex, there are $\binom{d+1}{2}$ compatible pairs of $d$-sets. By Lemma~\ref{thelemma}, the non-directed $d$-simplices each have at least $s(d)$ compatible pairs of $d$-sets. Therefore, there are at least 
\begin{equation}
\binom{d+1}{2}X + s(d)\left(\binom{n}{d+1}-X\right) \label{frombelow}
\end{equation}
compatible pairs of $d$-sets in $T_d$.
\\\\
On the other hand, let $A=\{a_1,a_2,\ldots,a_{d-1}\}$ be a set of $d-1$ vertices in $T_d$. We want to count the number of compatible pairs of $d$-sets that have $A$ as their common $(d-1)$-set. For $x\in T_d\setminus A$, the $d$-set $A\cup \{x\}$ is oriented one way or the other way. Let $h^+(A)$ be the number of vertices $x$ in $T_d\setminus A$ such that $A\cup\{x\}$ is oriented in one orientation and $h^-(A)$ is the number of vertices $y$ in $T_d\setminus A$ such that $A\cup\{y\}$ is oriented in the other orientation. The number of compatible pairs of $d$-sets that have $A$ as their common $(d-1)$-set is then $h^+(A)h^-(A)$, and as $h^+(A)+h^-(A)=n-d+1$ this value is maximised when $h^+(A)=\lfloor\frac{n-d+1}{2}\rfloor$. Therefore, there are at most
\begin{equation}
\binom{n}{d-1}\Big\lfloor\frac{n-d+1}{2}\Big\rfloor\Big\lceil\frac{n-d+1}{2}\Big\rceil \label{fromabove}
\end{equation}
compatible pairs of $d$-sets in $T_d$.
\\
Comparing \eqref{frombelow} and \eqref{fromabove}, we have
\begin{align*}
\binom{d+1}{2}X+s(d)\left(\binom{n}{d+1}-X\right)&\leq \binom{n}{d-1}\Big\lfloor\frac{n-d+1}{2}\Big\rfloor\Big\lceil\frac{n-d+1}{2}\Big\rceil \\
\text{so }\qquad X\left(\binom{d+1}{2}-s(d)\right)&\leq \left(\frac{d(d+1)}{4}-s(d)\right)\binom{n}{d+1}+O(n^d)\\
\text{so  }\quad \quad \qquad \qquad \qquad \qquad    X&\leq \left( \frac{\frac{d(d+1)}{4}-s(d)}{\frac{d(d+1)}{2}-s(d)}\right)\binom{n}{d+1}+O(n^d).
\end{align*}
When $d$ is odd, $s(d)=\frac{d+1}{2}\frac{d-1}{2}$ and 
\[ X\leq\frac{1}{d+1}\binom{n}{d+1}+O(n^d).\]
When $d$ is even, $s(d)=\frac{d^2}{4}$ and
\[ X\leq\frac{1}{d+2}\binom{n}{d+1}+O(n^d).\]
This completes the proof.
\end{proof}

\begin{remark}
From the proof above, the bound is tight only if there exists a $d$-tournament $T_d$ with the following properties.
\begin{enumerate}[(i)]
\item Almost all $d$-simplices in $T_d$ are either directed or have minimum number of compatible pairs of $d$-sets.
\item For almost all $(d-1)$-sets $A$, the number of compatible pairs of $d$-sets that have $A$ as their common $(d-1)$-set is about $\frac{n^2}{4}$.
\end{enumerate}
\end{remark}

\noindent Now we consider a lower bound on $c_d$.
\\\\
It is easy to see that a $d$-tournament where the orientation of each $d$-set is assigned randomly and independently with probability $\frac{1}{2}$ has expected number of directed $d$-simplices equal to $\frac{1}{2^d}\binom{n}{d+1}$. In other words, $c_d\geq \frac{1}{2^d}$. Unfortunately, we do not see how to generalise the construction for $d=3$ (inducing from a random 2-tournament) to obtain a significant improvement on $\frac{1}{2^d}$. In fact, we do not see how to use random methods to give a non-exponential lower bound on $c_d$. Instead, we give an explicit construction. Curiously, this construction gives $c_3\geq \frac{1}{7}$, which is worse than our previous method for $d=3$, and yet is much better in general.
\\\\
So, for the remainder of this section, we will give an explicit construction of a $d$-tournament that contains many directed $d$-simplices.

\begin{thm} \label{lowerbound}
Let $n=(d+1)^m$ where $m \in \mathbb{N}$ is sufficiently large. Then there is a $d$-tournament $T_d$ on $n$ vertices with at least 
\[ \frac{1}{1+\binom{d+1}{2}}\binom{n}{d+1}-o(n^{d+1}) \]
directed $d$-simplices. In particular, $c_d \geq \frac{1}{1+\binom{d+1}{2}}$.
\end{thm}

\begin{proof}
Let $D=\{1,2,\ldots,d+1\}$ be a fixed directed $d$-simplex: thus the $d$-sets of $D$ are oriented in a pairwise compatible way. 
\\\\
Let the vertex set of $T_d$ be $\{\mathbf{a}=(a_1,a_2,\ldots,a_m): 1 \leq a_i \leq d+1 \text{ for all }1 \leq i \leq m\}$.
Given a set of $d$ vertices $F=\{\mathbf{a}^{(1)},\mathbf{a}^{(2)},\ldots,\mathbf{a}^{(d)}\}$, set 
\[
j_F=\min \{1\leq j \leq m:a_j^{(1)},a_j^{(2)},\ldots,a_j^{(d)} \text{ all distinct}\} \quad \text{if it exists.}
\]
If $j_F$ exists, orient $F$ as the $d$-set $\{a_{j_F}^{(1)},a_{j_F}^{(2)},\ldots,a_{j_F}^{(d)}\}$ is oriented in $D$. If $j_F$ does not exist, orient $F$ in any way. Note that for a randomly chosen set of $d$ vertices $F$, we have $\mathbb{P}(j_F \text{ exists})\rightarrow 1$ as $m \rightarrow \infty$. That is to say, almost all $d$-sets get orientations according to the orientations of $d$-sets in $D$.
\\\\
Now we claim that $T_d$ has many directed $d$-simplices.
\\\\
Given a $d$-simplex $B=\{\mathbf{b}^{(1)},\mathbf{b}^{(2)},\ldots,\mathbf{b}^{(d+1)}\}$, set
\[
k_B=\min \{1 \leq k \leq m:b_k^{(1)},b_k^{(2)},\ldots,b_k^{(d+1)} \text{ all distinct}\} \quad \text{if it exists.}
\]
Then $B$ is a directed $d$-simplex if and only if $k_B$ exists and
\[
\left| \{b_k^{(1)}\} \cup \{b_k^{(2)}\} \cup \ldots \cup \{b_k^{(d+1)}\}\right|\leq d-1 \quad \text{for all }k < k_B.
\]
Let
\begin{align*}
             x&=\mathbb{P}(\text{a fixed coordinate of $d+1$ vectors has exactly $d+1$ distinct values})\\
              &=\frac{(d+1)!}{(d+1)^{d+1}},
\end{align*}
and
\begin{align*}
             y&=1-x-\mathbb{P}(\text{a fixed coordinate of $d+1$ vectors has exactly $d$ distinct values})\\
              &=1-\frac{(d+1)!}{(d+1)^{d+1}}-\frac{(d+1)d\binom{d+1}{2}(d-1)!}{(d+1)^{d+1}}\\
              &=1-\frac{(d+1)!}{(d+1)^{d+1}}\left(1+\binom{d+1}{2}\right).
\end{align*}
Now, choosing a random set of $d+1$ vectors is (as $m \rightarrow \infty$) asymptotically the same as choosing $d+1$ vectors with repetitions allowed. For the latter, the probability that these $d+1$ vectors form a directed $d$-simplex is exactly $x+xy+xy^2+\ldots +xy^{m-1}$. Thus, the fraction of directed $d$-simplices tends to $\frac{x}{1-y} = \frac{1}{1+\binom{d+1}{2}}$ as $m$ tends to infinity. The number of directed $d$-simplices in $T_d$ is therefore 
\[
\frac{1}{1+\binom{d+1}{2}}\binom{n}{d+1}-o(n^{d+1}).
\]
\end{proof}

\section{Remarks and Questions}

From Theorem~\ref{upperbound} and Theorem~\ref{lowerbound}, we have $\frac{1}{11} \leq c_4 \leq \frac{1}{6}$. Using a similar idea as the case $d=3$, that is, by inducing from a random 2-tournament, we can show that there is a 4-tournament on $n$ vertices which has at least $\frac{9}{64}\binom{n}{5}$ directed 5-simplices. (This is a messy case analysis, as there are many possible ways to induce a 4-tournament from a 2-tournament). Hence we have a better lower bound of $c_4\geq \frac{9}{64}$. However, as $d$ increases, the number of 2-tournaments on $d$ (or $d+1$) vertices increases rapidly. This makes the problem of finding the optimal way, or indeed any sensible way, of inducing a $d$-tournament from a 2-tournament hard. One could also try adapting this method to induce, for example, a 4-tournament from a 3-tournament, but unfortunately this seems no better than inducing a 4-tournament from a 2-tournament.
\\\\
Section 2 and the discussion above tell us that our construction in Theorem~\ref{lowerbound} for the lower bound of $c_d$ is not the best for the cases $d=3$ and $d=4$.

\begin{qs}
Is the construction of a $d$-tournament in Theorem~\ref{lowerbound} optimal, or nearly optimal, for $d\geq 5$?
\end{qs}

\noindent We have mentioned in the remark after Theorem~\ref{upperbound} that for our upper bound of $c_d$ to be tight, there must exist a very structured $d$-tournament. We believe the existence of such a $d$-tournament is fairly unlikely.

\begin{qs}
Can one improve the upper bound for $c_d$, for $d\geq 4$? What is the growth speed of $c_d$?
\end{qs}

\noindent Finally, we would like to see an example of an explicit 3-tournament attaining the exact value of $c_3$.

\begin{qs}
Is there a simple construction of a 3-tournament on $n$ vertices that has at least $\frac{1}{4} \binom{n}{4}+O(n^3)$ directed 3-simplices?
\end{qs}
\noindent By ``simple'', we mean not allowing a random construction or a quasi-random construction (such as a quadratic residue tournament).

\end{document}